\documentclass[oneside,english]{amsart}
\usepackage[T1]{fontenc}
\usepackage[latin9]{inputenc}
\usepackage{amsbsy}
\usepackage{amstext}
\usepackage{amsthm}
\usepackage{amssymb}
\usepackage{setspace}

\makeatletter
\allowdisplaybreaks

\author[L. Jiu]{Lin Jiu}
\address{Zu Chongzhi Center, Duke Kunshan University, Kunshan, Suzhou, Jiangsu Province, 215316, PR China}
\email[L. Jiu]{lin.jiu@dukekunshan.edu.cn, lin.jiu.work@gmail.com}

\author[Y. Yin]{Yihang Yin}
\address{Class of 2028, Duke Kunshan University, Kunshan, Suzhou, Jiangsu Province, 215316, PR China}
\email[Y. Yin]{yihang.yin@dukekunshan.edu.cn}


\makeatother

\providecommand\theoremname{Theorem}
\theoremstyle{plain}
\newtheorem{thm}{\protect\theoremname}
\providecommand\lemmaname{Lemma}
\newtheorem{lem}[thm]{\protect\lemmaname}
\providecommand\corollaryname{Corollary}
\newtheorem{cor}[thm]{\protect\corollaryname}
\providecommand\remarkname{Remark}
\theoremstyle{remark}
\newtheorem{rem}[thm]{\protect\remarkname}
\providecommand\examplename{Example}
\theoremstyle{definition}
\newtheorem{example}[thm]{\protect\examplename}
\providecommand\propositionname{Proposition}
\theoremstyle{plain}
\newtheorem{prop}[thm]{\protect\propositionname}
\usepackage{babel}
\begin{document}
\title[Hankel Determinants and J-Fractions of Bernoulli Numbers]{Partial Results on Hankel Determinants and the Corresponding J-Fractions
of a Sequence Related to Bernoulli Numbers}
\begin{abstract}
When exploring the Hankel determinant of the sequence $\mu_{k}=B_{k+1}/(k+1)$,
where $B_{k}$ is the $k$-th Bernoulli number, we obtained two interesting
results. The first one applies in general to all sequences $(c_{k})_{k\geq0}$
with all even-indexed term $0$, except for $c_{0}$. In this case,
the coefficient of the second highest order of the corresponding monic
orthogonal polynomials determines the Hankel determinants. Our second
result shows, the corresponding J-fractions, obtained from the generating
function of $\mu_{k}$, is exactly the same as in early work of Cao,
on a faster sequence converging to the Euler--Mascheroni constant
$\gamma$. 
\end{abstract}

\maketitle

\section{Introduction}

A \emph{Hankel matrix} $(M_{i,j})$ is a square matrix with constant
skew diagonals, i.e., $M_{i,j}=M_{i',j'}$ whenever $i+j=i'+j'$.
This terminology was named after Hermann Hankel. Given a sequence
$(c_{k})_{k\geq0}$ in a field $\mathbb{K}$, one may define the associated
Hankel matrices by $\ensuremath{(c_{i+j})_{0\le i,j\le n}}$ and may
also evaluate the determinant of these matrices. Such determinants
\[
H_{n}(c_{k}):=\underset{0\le i,j\le n}{\det}(c_{i+j})=\det\begin{pmatrix}c_{0} & c_{1} & c_{2} & \cdots & c_{n}\\
c_{1} & c_{2} & c_{3} & \cdots & c_{n+1}\\
c_{2} & c_{3} & c_{4} & \cdots & c_{n+2}\\
\vdots & \vdots & \vdots & \ddots & \vdots\\
c_{n} & c_{n+1} & c_{n+2} & \cdots & c_{2n}
\end{pmatrix}
\]
are called the \emph{Hankel determinants} of $(c_{k})_{k\geq0}$.
The choice of $(c_{k})_{k\geq0}$ relies on various reasons, but it
is usually of great interest to study classic number-theoretic sequences.
In recent years, extensive work has been devoted to evaluating Hankel
determinants for various sequences related to \emph{Bernoulli numbers}
$B_{k}$ (see, e.g., \cite{DilcherJiu} with a collection of related
results in the last section), which are defined by their exponential
generating function 
\begin{equation}
\frac{t}{e^{t}-1}=\sum^{\infty}_{k=0}B_{k}\frac{t^{k}}{k!}.\label{eq:DEFBernoulli}
\end{equation}
Al-Salam and Carlitz \cite[Eq.~(3.1), p.~93]{Al-SalamCarlitz} discovered
that 
\[
H_{n}(B_{k})=(-1)^{\binom{n+1}{2}}\prod^{n}_{j=1}\frac{(j!)^{6}}{(2j)!(2j+1)!}=(-1)^{\binom{n+1}{2}}\prod^{n}_{\ell=1}\left(\frac{\ell^{4}}{{4(2\ell+1)(2\ell-1)}}\right)^{n+1-\ell}.
\]
Note that when calculating $H_{n}(c_{k})$, it is important that the
sequence begins with $k=0$, as shifted sequences will result in related
but different Hankel determinants. For instance, for the shifted Bernoulli
numbers (see, e.g., \cite[Eq.~(3.57), p.~46]{K}),
\begin{align}
H_{n}(B_{k+1}) & =(-1)^{\binom{n+2}{2}}\frac{1}{2}\prod^{n}_{j=1}\frac{j!^{3}(j+1)!^{3}}{(2j+1)!(2j+2)!}\nonumber \\
 & =(-1)^{\binom{n+1}{2}}\left(-\frac{1}{2}\right)^{n+1}\prod^{n}_{\ell=1}\left(\left(\frac{\ell(\ell+1)}{2(2\ell+1)}\right)^{2}\right)^{n+1-\ell}.\label{eq:HnBk+1}
\end{align}

The calculation of $H_{n}(c_{k})$ mainly relies on the connection
to the orthogonal polynomials and continued fractions, with respect
to the given sequence $c_{k}$. Definition and more detailed results
will be introduced later in Section \ref{sec:Preliminaries}; here,
we note that, for instance, the corresponding monic orthogonal polynomials
with respect to $B_{k}$ are the continuous Hahn polynomials (see,
e.g., \cite[Thm.~8, p.~400]{JNT}). While for the shifted sequence
$B_{k+1}$, the first few corresponding monic orthogonal polynomials
are listed in Table \ref{tab:Q}. Some general results can be found
in, e.g., \cite{LeftShift} for shifts to the left, and \cite{DilcherJiu2}
for a shift to the right. 
\begin{table}
\begin{align*}
Q_{0}(y) & =1\\
Q_{1}(y) & =y+\frac{1}{3}\\
Q_{2}(y) & =y^{2}+\frac{3}{5}y+\frac{1}{5}\\
Q_{3}(y) & =y^{3}+\frac{6}{7}y^{2}+\frac{5}{7}y+\frac{6}{35}.
\end{align*}

\caption{\label{tab:Q} The Monic Orthogonal Polynomials $Q_{n}(y)$ with Respect
to $B_{k+1}$, for $0\le n\protect\leq3$. }
\end{table}

Evaluating Hankel determinants often presents a subtle challenge due
to the sensitivity to modifications of the underlying sequence. Simple
operations, such as passing from $a_{k}$ to $ka_{k}$ or $a_{k+1}/(k+1)$,
can drastically alter the corresponding Hankel determinants. For example,
in the collection \cite{DilcherJiu} of Hankel determinants of sequences
related to Bernoulli and Euler polynomials, such cases are independently
obtained. Therefore, in the current work, we focus on one missing
piece: the Hankel determinants of the sequence 
\[
\mu_{k}:=\frac{B_{k+1}}{k+1}.
\]
Introducing the linear denominator $1/(k+1)$, after shifting the
index, breaks the classic hypergeometric framework of continuous Hahn
polynomials associated with $B_{k}$. As a result, standard tools
fail to yield explicit determinant evaluations directly; instead,
we obtained two interesting and unexpected results. 

The first one is a general pattern, which applies to all sequences
$c_{n}$ satisfying $c_{0}\neq0$, and $c_{2m}=0$ for all $m\in\mathbb{N}$,
rather than merely to $\mu_{n}$. For example, if we concentrate on
the second highest coefficients in Table \ref{tab:Q}, the sequence
\[
\left(d_{\ell}\right)_{\ell\geq1}=\left(\frac{1}{3},\frac{3}{5},\frac{6}{7},\ldots\right)
\]
admits the formula 
\[
d_{\ell}=\frac{\ell(\ell+1)}{2(2\ell+1)},
\]
which appears in the product of \eqref{eq:HnBk+1}, not coincidentally;
moreover, certain recurrences formula exists, in this case. Note that
\[
Q_{3}(y)=\left(y+d_{3}-d_{2}\right)Q_{2}(y)+d^{2}_{2}Q_{1}(y).
\]
Details on the statement and proof will be presented in Subsection
\ref{subsec:GeneralParity}. Such a connection has been overlooked, partly because explicit expressions, e.g., \eqref{eq:HnBk+1},are often favored over partial patterns. 

The second result reveals that the generating function of $\mu_{k}$,
which admits a continued fraction expression by the classic results,
is exactly the same in early work of Cao \cite{Cao}, on the approximation
of the the \emph{Euler-Mascheroni constant} $\gamma$. To our surprise,
the two continued fractions are exactly the same; although they independently
originated from different goals and approaches. 

\textbf{Outline of this paper}. In Section \ref{sec:Preliminaries},
we will introduce necessary preliminaries, e.g., the deep connection
among Hankel determinants, orthogonal polynomials, and continued fractions,
that will be used in Section \ref{sec:Main-Results}. Then, we divide
our main results into three subsections. In Subsection \ref{subsec:GeneralParity},
we will state and prove the general result that, for certain sequences,
their Hankel determinants and orthogonal polynomials are determined
only by the second highest coefficients. To apply this result to our
main object $\mu_{k}=B_{k+1}/(k+1)$, the nondegeneracy condition
that $H_{n}(\mu_{k+r})\neq0$ for all $n=0,1,\ldots$ and $r=0,1$,
is proven in Subsection \ref{subsec:Nondegeneracy}. Finally in Subsection
\ref{subsec:ContinuedFractions}, we will show the continued fraction
expression of the generating function of $\mu_{k}$ is the same as
that established by Cao \cite{Cao}, to obtain a faster convergent
sequence to $\gamma$. 

\section{Preliminaries\label{sec:Preliminaries}}

The study of Hankel determinants has been extensively developed, mainly
based on the close relationship to classical orthogonal polynomials
and continued fractions; see, e.g., \cite{Chihara}, \cite{Ismail}
and \cite{K}. Here, we state some basic results for later use, in
order to make this paper self-contained. 

Suppose we are given a sequence $(c_{k})_{k\geq0}$ of numbers; then
we can define a linear functional $L$ on polynomials by, for $k=0,1,\ldots$,
\begin{equation}
L(y^{k})=c_{k}.\label{eq:LinearOperator}
\end{equation}

\begin{lem}
Let $L$ be the linear functional in \eqref{eq:LinearOperator}. If
(and only if) $H_{n}(c_{k})\neq0$ for all $n=0,1,2\ldots$, there
exists a unique sequence of monic polynomials $P_{n}(y)$ of degree
$n$ and a sequence of nonzero numbers $\zeta_{n}$ such that 
\[
L\left(P_{m}(y)P_{n}(y)\right)=\zeta_{n}\delta_{m,n}
\]
where $\delta_{m,n}$ denotes the Kronecker delta function that is equal to $1$ if $m=n$; and $0$ otherwise. Furthermore, for all $n\in\mathbb{N}$,
we have $\zeta_{n}=H_{n}(c_{k})/H_{n-1}(c_{k})$, and 
\begin{equation}
P_{n}(y)=\frac{1}{H_{n-1}(c_{k})}\det\begin{pmatrix}c_{0} & c_{1} & \cdots & c_{n}\\
c_{1} & c_{2} & \cdots & c_{n+1}\\
\vdots & \vdots & \ddots & \vdots\\
c_{n-1} & c_{n} & \cdots & c_{2n-1}\\
1 & y & \cdots & y^{n}
\end{pmatrix}.\label{eq:DEFPn}
\end{equation}
And the polynomials $P_{n}(y)$ satisfy the three-term recurrence
relation that $P_{0}(y)=1$, $P_{1}(y)=y+s_{0}$, and for $n\geq1$,
\begin{equation}
P_{n+1}(y)=(y+s_{n})P_{n}(y)+t_{n}P_{n-1}(y),\label{eq:3TermRec}
\end{equation}
for some sequences $(s_{n})_{n\geq0}$ and $(t_{n})_{n\geq1}$. 
\end{lem}

If we multiply both sides of \eqref{eq:DEFPn} by $y^{r}$ and replace
$y^{j}$ by $c_{j}$ , which includes replacing the constant term
$1$ by $c_{0}$, then the following corollary holds.
\begin{cor}
With the sequence $(c_{k})_{k\geq0}$ and the polynomials $P_{n}(y)$
in \eqref{eq:DEFPn}, we have 
\begin{equation}
y^{r}P_{n}(y)\bigg{|}_{y^k=c_k}=\begin{cases}
0, & 0\leq r\leq n-1;\\
\zeta_{n}, & r=n.
\end{cases}\label{eq:yrPnyEval}
\end{equation}
\end{cor}

\begin{rem}
Since $P_{n}(y)$ are monic, the $n$ equations, obtained by letting $r$
run from $0$ to $n-1$ in \eqref{eq:yrPnyEval} determine $P_{n}(y)$.
This is equivalent to the definition \eqref{eq:DEFPn}. Now, we can
have the formula for the Hankel determinants of $c_{k}$. 
\end{rem}

\begin{cor}
With the sequence $(t_{n})_{n\geq1}$ as in \eqref{eq:3TermRec},
\[
H_{n}(c_{k})=c^{n+1}_{0}\prod^{n}_{\ell=1}(-t_{\ell})^{n+1-\ell}=(-1)^{\binom{n+1}{2}}c^{n+1}_{0}t^{n}_{1}t^{n-1}_{2}\cdots t_{n}.
\]
\end{cor}

The next result shows the connection between the sequence $(c_{k})_{k\geq0}$
and a continued fraction expression, which is called the \emph{Jacobi
fractions}, or \emph{J-fractions}. Note that we require $H_{n}(c_{k})\neq0$,
which includes $c_{0}=H_{0}(c_{k})\neq0$. 
\begin{cor}
For a variable $x$, 
\begin{equation}
\sum^{\infty}_{k=0}c_{k}x^{k}=\frac{c_{0}}{1+s_{0}x+\frac{t_{1}x^{2}}{1+s_{1}x+\frac{t_{2}x^{2}}{1+s_{2}x+\ddots}}}=\frac{c_{0}}{1+s_{0}x+}\underset{j=1}{\stackrel{\infty}{\mathbf{K}}}\frac{t_{j}x^{2}}{1+s_{j}x},\label{eq:JCF}
\end{equation}
where we adopt the notation of the continued fractions 
\[
\underset{j=1}{\stackrel{\infty}{\mathbf{K}}}\frac{a_{j}}{b_{j}}=\frac{a_{1}}{b_{1}+\underset{j=2}{\stackrel{\infty}{\mathbf{K}}}\frac{a_{j}}{b_{j}}}=\frac{a_{1}}{b_{1}+}\underset{j=2}{\stackrel{\infty}{\mathbf{K}}}\frac{a_{j}}{b_{j}}=\frac{a_{1}}{b_{1}+\frac{a_{2}}{b_{2}+\ddots}},
\]
and its $n$th approximants 
\[
\underset{j=1}{\stackrel{n}{\mathbf{K}}}\frac{a_{j}}{b_{j}}=\frac{a_{1}}{b_{1}+\frac{a_{2}}{b_{2}+\ddots+\frac{a_{n}}{b_{n}}}}.
\]
\end{cor}

The next lemma is about determinants of ``checkerboard matrices'',
namely matrices in which every other entry vanishes. This result can
be found in \cite[Lemmas 5 and 6]{Checker} and covers more general
matrices than just Hankel matrices.
\begin{lem}
\label{lem:Checker}Let $M=(M_{i,j})_{0\leq i,j\leq n-1}$ be a matrix. 
\begin{enumerate}
\item If $M_{i,j}=0$ whenever $i+j$ is odd, then
\begin{equation}
\det_{0\leq i,j\leq n-1}(M_{i,j})=\det_{0\leq i,j\leq\lfloor\frac{n-1}{2}\rfloor}(M_{2i,2j})\cdot\det_{0\leq i,j\leq\lfloor\frac{n-2}{2}\rfloor}(M_{2i+1,2j+1}).\label{eq:CheckerOdd0}
\end{equation}
\item If $M_{i,j}=0$ whenever $i+j$ is even, then
\begin{enumerate}
\item for an even $n$, 
\begin{equation}
\det_{0\leq i,j\leq n-1}(M_{i,j})=(-1)^{\frac{n}{2}}\det_{0\leq i,j\leq\lfloor\frac{n-1}{2}\rfloor}(M_{2i+1,2j})\cdot\det_{0\leq i,j\leq\lfloor\frac{n-2}{2}\rfloor}(M_{2i,2j+1});\label{eq:CheckerEven01}
\end{equation}
\item while for an odd $n$ 
\begin{equation}
\det_{0\leq i,j\leq n-1}(M_{i,j})=0.\label{eq:CheckerEven02}
\end{equation}
\end{enumerate}
\end{enumerate}
\end{lem}

\begin{rem}
The \emph{floor function} $\lfloor\cdot\rfloor$ is used in the previous
lemma. We also note that \eqref{eq:CheckerOdd0}---\eqref{eq:CheckerEven02}
are best explained through some examples. 
\end{rem}

\begin{example}
By \eqref{eq:CheckerOdd0}, we have
\[
\det\begin{pmatrix}a & 0 & b & 0 & c\\
0 & \boldsymbol{d} & 0 & \boldsymbol{e} & 0\\
f & 0 & g & 0 & h\\
0 & \boldsymbol{i} & 0 & \boldsymbol{j} & 0\\
k & 0 & \ell & 0 & m
\end{pmatrix}=\det\begin{pmatrix}a & b & c\\
f & g & h\\
k & \ell & m
\end{pmatrix}\cdot\det\begin{pmatrix}\boldsymbol{d} & \boldsymbol{e}\\
\boldsymbol{i} & \boldsymbol{j}
\end{pmatrix},
\]
and 
\[
\det\begin{pmatrix}a & 0 & b & 0\\
0 & \boldsymbol{d} & 0 & \boldsymbol{e}\\
f & 0 & g & 0\\
0 & \boldsymbol{i} & 0 & \boldsymbol{j}
\end{pmatrix}=\det\begin{pmatrix}a & b\\
f & g
\end{pmatrix}\cdot\det\begin{pmatrix}\boldsymbol{d} & \boldsymbol{e}\\
\boldsymbol{i} & \boldsymbol{j}
\end{pmatrix}.
\]
Also, by \eqref{eq:CheckerEven01} and \eqref{eq:CheckerEven02},
we see
\[
\det\begin{pmatrix}0 & \boldsymbol{d} & 0 & \boldsymbol{e}\\
f & 0 & g & 0\\
0 & \boldsymbol{i} & 0 & \boldsymbol{j}\\
k & 0 & \ell & 0
\end{pmatrix}=\det\begin{pmatrix}f & g\\
k & \ell
\end{pmatrix}\cdot\det\begin{pmatrix}\boldsymbol{d} & \boldsymbol{e}\\
\boldsymbol{i} & \boldsymbol{j}
\end{pmatrix},
\]
and 
\[
\det\begin{pmatrix}0 & \boldsymbol{d}\\
a & 0
\end{pmatrix}=-a\cdot\boldsymbol{d},\quad\det\begin{pmatrix}0 & \boldsymbol{d} & 0\\
f & 0 & g\\
0 & \boldsymbol{i} & 0
\end{pmatrix}=0.
\]
\end{example}

The following classic results of the sequence with all odd terms $0$,
i.e., being \emph{symmetric}, can be found in, e.g., \cite[Thm.~4.3, p.~21]{Chihara}.
One can alternatively show it by combining \eqref{eq:DEFPn} and Lem.~\ref{lem:Checker}. 
\begin{thm}
If $(c_{k})_{k\geq0}$ has all odd-indexed terms zero, i.e., $c_{2m-1}=0$
for all $m\in\mathbb{N}$, then, its corresponding monic orthogonal
polynomials with recurrence in \eqref{eq:3TermRec} satisfy $P_{n}(-y)=(-1)^{n}P_{n}(y)$
and $s_{n}=0$. 
\end{thm}

Finally, from the exponential generating function (and the definition) \eqref{eq:DEFBernoulli} of the Bernoulli numbers $B_{k}$, we have one important property: 
$\mu_{k+1}$ (as well as $B_{k+2}$) are symmetric. 
\begin{thm}
$B_{2n+1}=0$ for all $n\in\mathbb{N}$, and $B_{1}=-1/2$ is the
only nonzero, odd-indexed Bernoulli number. 
\end{thm}

\section{Main Results\label{sec:Main-Results}}

\subsection{General Results on Hankel Determinants of Special Sequences\label{subsec:GeneralParity}}

In this subsection, we consider a general sequence $(c_{k})_{k\geq0}$
with the following properties: 
\begin{enumerate}
\item $c_{0}\neq0$, and $c_{2m}=0$ for all $m\in\mathbb{N}$ (namely, 
the shifted sequence $(c_{k+1})_{k\geq0}$ is symmetric); 
\item and the shifted Hankel determinants of $c_{k+1}$ are nonzero, i.e.,
for all $n=0,1,2\ldots,$ $H_{n}(c_{k+1})\neq0$. 
\end{enumerate}
We further let $P_{n}(y)$ be the corresponding monic orthogonal polynomials
with respect to $c_{k}$, whose three-term recurrence is given in
\eqref{eq:3TermRec}. Consider the expansion of $P_{n}(y)$ as
\[
P_{n}(y)=y^{n}+d_{n}y^{n-1}+\text{lower terms}.
\]
It is generally true, regardless of the properties of $c_{k}$, that
\begin{equation}
s_{n}=d_{n+1}-d_{n},\label{eq:sndn}
\end{equation}
by matching coefficients of $y^{n}$ on both sides of \eqref{eq:3TermRec}.
Now, with the properties mentioned above, we can further express $t_{n}$
in terms of $d_{n}$. 
\begin{thm}
\label{thm:tndn}We have $t_{n}=d^{2}_{n}$. 
\end{thm}

\begin{proof}
Define the sequence $\lambda_{n}:=P_{n+1}(0)/P_{n}(0)$. Note that
by \eqref{eq:DEFPn}, 
\[
P_{n}(0)=(-1)^{n}\frac{H_{n-1}(c_{k+1})}{H_{n-1}(c_{k})}\neq0,
\]
so that $\lambda_{n}\neq0$ is well-defined. Then consider 
\[
\tilde{P}_{n}(y):=\frac{P_{n+1}(y)-\lambda_{n}P_{n}(y)}{y},
\]
which is also a polynomial in $y$ of degree $n$, by the definition
of $\lambda_{n}$. Now we claim that $\tilde{P}_{n}(y)$ is the orthogonal
polynomial with respect to the shifted sequence $c_{k+1}$. By \eqref{eq:yrPnyEval},
it suffices to show that for any $0\leq r\leq n-1$, 
\[
\tilde{P}_{n}(y)y^{r}\bigg|_{y^{k}=c_{k+1}}=0.
\]
A direct calculation shows that 
\begin{align*}
\tilde{P}_{n}(y)y^{r}\bigg|_{y^{k}=c_{k+1}} & =\tilde{P}_{n}(y)y^{r+1}\bigg|_{y^{k}=c_{k}}=\left(P_{n+1}(y)-\lambda_{n}P_{n}(y)\right)y^{r}\bigg|_{y^{k}=c_{k}}\\
 & =P_{n+1}(y)y^{r}\bigg|_{y^{k}=c_{k}}-\lambda_{n}P_{n}(y)y^{r}\bigg|_{y^{k}=c_{k}}=0,
\end{align*}
where in the last two steps, we applied the linearity. Note that for
the shifted sequence $\tilde{c}_{k}=c_{k+1}$, it is symmetric. Hence,
$\tilde{P}_{n}(y)$ is either odd or even, depending on the parity
of $n$. In particular, the coefficient of $y^{n-1}$ must be $0$.
This, by definition, yields 
\begin{equation}
0=d_{n+1}-\lambda_{n}=d_{n+1}-\frac{P_{n+1}(0)}{P_{n}(0)}.\label{eq:dnlambdan}
\end{equation}
Let $y=0$ in \eqref{eq:3TermRec} to see 
\[
P_{n+1}(0)=s_{n}P_{n}(0)+t_{n}P_{n-1}(0),
\]
which, by applying \eqref{eq:sndn} and \eqref{eq:dnlambdan}, implies
\[
d_{n+1}=(d_{n+1}-d_{n})+\frac{t_{n}}{d_{n}},
\]
so that $t_{n}=d^{2}_{n}$. 
\end{proof}

\begin{cor}
Recall that $s_{n}=d_{n+1}-d_{n}$, so $P_{n}(y)$ are completely
determined by the sequence $d_{n}$ via the recurrence 
\[
P_{n+1}(y)=\left(y+d_{n+1}-d_{n}\right)P_{n}(y)+d^{2}_{n}P_{n-1}(y).
\]
\end{cor}

\begin{example}
We can construct certain sequences from the sequence $d_{n}$, by
first letting $s_{n}=d_{n+1}-d_{n}$ and $t_{n}=d^{2}_{n}$. Then, the polynomials
$P_{n}(y)$ and the corresponding sequence $c_{k}$ follow naturally. 
\begin{enumerate}
\item If we let $d_{n}=1$ except for $d_{0}=0$, then $s_{0}=1$ but $s_{n}=0$
for all $n\in\mathbb{N}$. We will have the sequence 
\[
(c_{k})_{k\geq0}=(1,-1,0,1,0,-2,0,5,0,-14,0,42,\ldots).
\]
Namely, $c_{0}=1$ and, for all $m\in\mathbb{N}$, $c_{2m}=0$ and 
\[
c_{2m-1}=(-1)^{m}C_{m}=\frac{(-1)^{m}}{m+1}\binom{2m}{m},
\]
where $C_{m}$ is the $m$-th Catalan number. 
\item Another choice is $d_{n}=n$, so that $s_{n}=1$ and $t_{n}=n^{2}$.
It is not hard to check, in this case 
\[
c_{2m-1}=\frac{2^{2m}(1-2^{2m})B_{2m}}{2m}.
\]
\end{enumerate}
\end{example}

\subsection{Nondegeneracy of $\mu_{n}$\label{subsec:Nondegeneracy}}

By the integral representation (, see, e.g., \cite[Entry 24.7.2]{DLMF})
that 
\[
B_{2k}=(-1)^{k+1}4k\int^{\infty}_{0}\frac{x^{2k-1}}{e^{2\pi x}-1}dx,
\]
we define, for $k=0,1,2,\ldots,$
\[
\beta_{k}:=(-1)^{k}\frac{B_{2k+2}}{2k+2}=2\int^{\infty}_{0}\frac{x^{2k+1}}{e^{2\pi x}-1}dx.
\]
The change of variables $t=x^2$ leads to 
\[
\beta_{k}=\int^{\infty}_{0}t^{k}d\rho(t),
\]
where 
\[
d\rho(t)=\frac{dt}{e^{2\pi\sqrt{t}}-1}.
\]
Or, we alternatively obtain 
\[
\rho(t)=\frac{\sqrt{t}\log\left(1-e^{-2\pi\sqrt{t}}\right)}{\pi}-\frac{1}{2\pi^{2}}\mathrm{Li}_{2}\left(e^{-2\pi\sqrt{t}}\right),
\]
where the \emph{dilogarithm function} $\mathrm{Li}_{2}(x)$ is defined by
\[
\mathrm{Li}_{2}(x):=\sum^{\infty}_{k=1}\frac{x^{k}}{k^{2}}.
\]
\begin{lem}
The $(n+1)\times(n+1)$ shifted Hankel matrix, for $r=0,1,\ldots$,
\[
G^{(r)}_{n}:=\left(\beta_{i+j+r}\right)_{0\leq i,j\leq n}
\]
is positive definite. 
\end{lem}

\begin{proof}
Symmetry is guaranteed automatically by the definition of $G^{(r)}_{n}$.
Consider any nonzero column vector 
\[
\vec{u}:=(u_{0},u_{1},\ldots,u_{n})^{T}\in\mathbb{R}^{n+1},
\]
and the associated (non-zero) polynomial 
\[
p_{\vec{u}}(t):=u_{0}+u_{1}t+\cdots+u_{n}t^{n}.
\]
Observe that 
\begin{align*}
\vec{u}^{T}G^{(r)}_{n}\vec{u} & =\sum^{n}_{i,j=0}u_{i}u_{j}\beta_{i+j+r}=\sum^{n}_{i,j=0}u_{i}u_{j}\int^{\infty}_{0}t^{i+j+r}d\rho(t)\\
 & =\int^{\infty}_{0}t^{r}\sum^{n}_{i,j=0}u_{i}u_{j}t^{i+j}d\rho(t)=\int^{\infty}_{0}t^{r}\left(p_{\vec{u}}(t)\right)^{2}d\rho(t).
\end{align*}
On the interval $(0,\infty)$ all three pieces $t^{r}$, $p_{\vec{u}}(t)^{2}$,
and $d\rho$ of the integrand are nonnegative everywhere and strictly positive on a small interval. This shows 
\[
\vec{u}^{T}G^{(r)}_{n}\vec{u}>0.\qedhere
\]
\end{proof}

Now, we denote 
\[
\eta_{k}:=(-1)^{k}\beta_{k}=\frac{B_{2k+2}}{2k+2}=\mu_{2k+1}.
\]
If we let, for $r=0,1,\ldots$, 
\[
A^{(r)}_{n}:=\left(\eta_{i+j+r}\right)_{0\leq i,j\leq n}
\]
be the shifted Hankel matrix and consider the diagonal matrix
\[
D_{n}:=\mathrm{diag}(1,-1,1,-1,\ldots,(-1)^{n}),
\]
we see 
\[
A^{(r)}_{n}=(-1)^{r}D_{n}G^{(r)}_{n}D_{n}.
\]
And consequently, 
\[
\det A^{(r)}_{n}=(-1)^{r(n+1)}\det(G^{(r)}_{n})\neq0.
\]
Now, we are ready to show two Hankel determinants are nonzero: $H_{n}(\mu_{k})$
and $H_{n}(\mu_{k+1})$, where the former is used to guarantee the existence
of orthogonal polynomials, without knowing their explicit expressions;
while the latter is one of the conditions required in Thm.~\ref{thm:tndn},
which is easier to prove. 
\begin{prop}
For all $n=0,1,2,\ldots,$ $H_{n}(\mu_{k+1})\neq0.$
\end{prop}

\begin{proof}
The shifted sequence $\mu_{k+1}=B_{k+2}/(k+2)$ has all odd-indexed terms being
$0$. Then we begin with $H_n(\mu_{k+1})=B_2/2\neq 0$. For $n\geq 1$, by \eqref{eq:CheckerEven01}, we have 
\[
H_{2n+1}(\mu_{k+1})=\det A^{(0)}_{n}\det A^{(1)}_{n}\neq0,
\]
and 
\[
H_{2n}(\mu_{k+1})=\det A^{(0)}_{n}\det A^{(1)}_{n-1}\neq0.\qedhere
\]
\end{proof}

\begin{prop}
For all $n=0,1,2,\ldots,$$H_{n}(\mu_{k})\neq0$. 
\end{prop}

\begin{proof}
Let 
\[
M_{n}:=\left(\mu_{i+j}\right)_{0\leq i,j\leq n}
\]
be the Hankel matrix of order $n+1$; and denote
\[
S_{n}=\mathrm{diag}\left(\mu_{0},\underset{n\,\text{copies}}{\underbrace{0,0,\ldots,0}}\right).
\]
\begin{enumerate}
\item For the odd case that $n=2m+1$ case, $M_{n}$ has even order of $2m+2$.
After simultaneously permuting rows and columns, it becomes
\[
\bar{M}_{n}=\begin{pmatrix}S_{m} & A^{(0)}_{m}\\
A^{(0)}_{m} & \mathbf{0}_{(m+1)\times(m+1)}
\end{pmatrix},
\]
and such permutations only create an extra factor of either $1$ or $-1$, for which we denoted by $(\pm)$. Then,  
\[
H_{2m+1}(\mu_{k})=\det M_{2m+1}=(\pm)\left(\det A^{(0)}_{m}\right)^{2}\neq0.
\]
\item For the even case that $n=2m$, similar permutations lead to
\[
\tilde{M}_{n}=\begin{pmatrix}\mu_{0} & 0 & \vec{v}^{T}_{m}\\
0 & \mathbf{0}_{m\times m} & A^{(1)}_{m-1}\\
\vec{v}_{m} & A^{(1)}_{m-1} & \mathbf{0}_{m\times m}
\end{pmatrix},
\]
where 
\[
\vec{v}_{m}=\left(\mu_{1},\mu_{3},\ldots,\mu_{2m-1}\right)^{T}=\left(\eta_{0},\eta_{1},\ldots,\eta_{m-1}\right)^{T}.
\]
Consider the following $(2m+1)\times(2m+1)$ upper triangular matrix
\[
L_{m}:=\begin{pmatrix}1 & -\vec{v}^{T}_{m}\left(A^{(1)}_{m-1}\right)^{-1} & \mathbf{0}_{1\times m}\\
\mathbf{0}_{m\times1} & I_{m} & \mathbf{0}_{m\times m}\\
\mathbf{0}_{m\times1} & \mathbf{0}_{m\times m} & I_{m}
\end{pmatrix},
\]
where $I_{m}$ is the $m$ by $m$ identities matrix. 
\begin{align*}
L_{m}\tilde{M}_{n} & =\begin{pmatrix}\mu_{0} & \mathbf{0}_{1\times m} & \vec{v}^{T}_{m}-\vec{v}^{T}_{m}\left(A^{(1)}_{m-1}\right)^{-1}A^{(1)}_{m-1}\\
\mathbf{0}_{m\times1} & \mathbf{0}_{m\times m} & A^{(1)}_{m-1}\\
\vec{v}_{m} & A^{(1)}_{m-1} & \mathbf{0}_{m\times m}
\end{pmatrix}\\
 & =\begin{pmatrix}\mu_{0} & \mathbf{0}_{1\times m} & \mathbf{0}_{1\times m}\\
\mathbf{0}_{m\times1} & \mathbf{0}_{m\times m} & A^{(1)}_{m-1}\\
\vec{v}_{m} & A^{(1)}_{m-1} & \mathbf{0}_{m\times m}
\end{pmatrix}.
\end{align*}
Since $\det L_{m}=1$, we have 
\begin{align*}
\det\tilde{M}_{n} & =\det\begin{pmatrix}\mu_{0} & \mathbf{0}_{1\times m} & \mathbf{0}_{1\times m}\\
\mathbf{0}_{m\times1} & \mathbf{0}_{m\times m} & A^{(1)}_{m-1}\\
\vec{v}_{m} & A^{(1)}_{m-1} & \mathbf{0}_{m\times m}
\end{pmatrix}\\
 & =\mu_{0}\det\begin{pmatrix}\mathbf{0}_{m\times m} & A^{(1)}_{m-1}\\
A^{(1)}_{m-1} & \mathbf{0}_{m\times m}
\end{pmatrix}=(\pm)\mu_{0}\left(\det A^{(1)}_{m-1}\right)^{2}\neq0.
\end{align*}
Hence, 
\[
\det M_{n}=(-1)^{m^{2}}\det\tilde{M}_{n}\neq0.\qedhere
\]
\end{enumerate}
\end{proof}

We can now apply Thm.~\ref{thm:tndn} to $\mu_{k}$. 
\begin{cor}
If $R_{n}(y)=y^{n}+\bar{d}_{n}y^{n-1}+\cdots$ is the monic orthogonal
polynomial of degree $n$ with respect to $\mu_{k}$, then
\[
H_{n}(\mu_{k})=(-1)^{\binom{n+1}{2}}\left(-\frac{1}{2}\right)^{n+1}\prod^{n}_{\ell=1}\bar{d}^{2(n+1-\ell)}_{\ell}.
\]
\end{cor}

\subsection{Two Continued Fractions for $\gamma$\label{subsec:ContinuedFractions}}

As defined above that $R_{n}(y)$ is the monic orthogonal polynomial of degree $n$
with respect to $\mu_{k}$. To make the indices compatible, we assume
the three-term recurrence of $R_{n}$ are given by 
\begin{equation}
R_{n+1}(y)=(y+\beta_{n+1})R_{n}(y)+\alpha_{n+1}R_{n-1}(y),\label{eq:3TermRn}
\end{equation}
for $n\geq 1$. Then,by letting $\alpha_1=-\mu_0 =1/2$ and
the J-fractions \eqref{eq:JCF},
\[
\sum^{\infty}_{k=0}\mu_{k}x^{k}=-\frac{\alpha_{1}}{1+\beta_{1}x+}\underset{j=2}{\stackrel{\infty}{\mathbf{K}}}\frac{\alpha_{j}x^{2}}{1+\beta_{j}x}.
\]
For comparison, Cao \cite{Cao} also has $a_1=-1/2=\alpha_1$. 

On the other hand, Cao \cite[Thm.~5, p.~1441]{Cao} refined the work
of Lu \cite{Lu} and Xu and You \cite{XuYou}, to obtain a continued
fraction approximation to the Euler-Mascheroni constant 
\[
\gamma:=\lim_{n\rightarrow\infty}\left(H_{n}-\log n\right).
\]
More precisely, define 
\[
E_{k}(n):=H_{n}-\log n-\gamma-MC_{k}(n),
\]
where 
\[
MC_{k}(n)=\begin{cases}
\frac{a_{1}}{n+b_{1}}, & k=1;\\
\frac{a_{1}}{n+b_{1}+}\underset{j=2}{\stackrel{k}{\mathbf{K}}}\frac{a_{j}}{n+b_{j}}, & k\geq2,
\end{cases}
\]
and the first few terms of the two sequences, $a_{k}$ and $b_{k}$,
are listed as follows 

\begin{onehalfspace}
\[
\begin{array}{ll}
a_{1}=\frac{1}{2}, & b_{1}=\frac{1}{6},\\
a_{2}=\frac{1}{36}, & b_{2}=\frac{13}{30},\\
a_{3}=\frac{9}{25}, & b_{3}=\frac{17}{630},\\
a_{4}=\frac{6241}{15876}, & b_{4}=\frac{417941}{786366},\\
a_{5}=\frac{52272900}{38950081}, & b_{5}=-\frac{1835967509}{23923912386},\\
a_{6}=\frac{17194548650161}{14694541555716}, & b_{6}=\frac{431312596940299603}{686480136010816290},\\
a_{7}=\frac{93778512198179213368089}{32070070056327569608225}, & b_{7}=-\frac{75178865368857369613934863}{437108607837436422694763190},\\
a_{8}=\frac{14093175882028689333655328914081}{5957702453097198927838844740836}, & b_{8}=\frac{152838545298199920648591716358691154137}{212256305311307139071033336233757302422}.
\end{array}
\]

\end{onehalfspace}

\noindent Then, for all $k\in\mathbb{N}$, 
\begin{align}
\lim_{n\rightarrow\infty}n^{2k+2}\left(E_{k}(n)-E_{k}(n+1)\right) & =(2k+1)C_{k},\nonumber \\
\lim_{n\rightarrow\infty}n^{2k+1}E_{k}(n) & =C_{k},\label{eq:CaoAppr}
\end{align}
for a sequence of constants $C_{k}$. The first eight terms of $C_{k}$
are also given in \cite[Thm.~5, p.~1441]{Cao}. This series of work
provides faster convergent sequences to certain constants, including
$\gamma$. However, no explicit formulas for $a_{k}$, $b_{k}$, or
$C_{k}$ are given. The goal of this subsection is to show that for $j\geq 1$ $a_{j+1}=\alpha_{j+1}$
and $b_{j}=\beta_{j}$, where $\alpha_{k}$ and $\beta_{n}$ in \eqref{eq:3TermRn}. 

Consider the digamma function
\[
\psi(x):=\frac{d}{dx}\left(\log\Gamma(x)\right)=\frac{\Gamma'(x)}{\Gamma(x)},
\]
which satisfies \cite[(6.3.2), p.~258]{Handbook} $\psi(1)=-\gamma$
and $\psi(n)=-\gamma+H_{n-1}$ for $n\geq2$. Applying the Euler-Maclaurin
summation formula to $f(x)=1/x$ on $[1,n]$, one can obtain the asymptotic
\cite[(6.3.18), p.~259]{Handbook}
\begin{equation}
\psi(x)\sim\log x-\frac{1}{2x}-\sum^{\infty}_{k=1}\frac{B_{2k}}{2k\cdot x^{2k}}=\log x-\frac{1}{2x}-\frac{1}{12x^{2}}+\frac{1}{120x^{4}}-\frac{1}{252x^{6}}+\cdots,\label{eq:Gamma_Euler_Maclaurin}
\end{equation}
as $x\rightarrow\infty$ in $\left|\arg x\right|<\pi$. Letting $x=n$
in \eqref{eq:Gamma_Euler_Maclaurin}, we have 
\[
H_{n-1}-\gamma\sim\log n-\frac{1}{2n}-\sum^{\infty}_{k=1}\frac{B_{2k}}{2k\cdot n^{2k}},
\]
or, more precisely, for $N\in\mathbb{N}$ and as $n\rightarrow\infty$,
\[
H_{n}-\log n-\gamma-\left(\frac{1}{2n}-\sum^{N}_{m=1}\frac{B_{2m}}{2m\cdot n^{2m}}\right)=O\left(n^{-2N-2}\right).
\]
Let 
\[
\Phi(x):=\frac{x}{2}-\sum^{\infty}_{m=1}\frac{B_{2m}}{2m}x^{2m}=-x\sum^{\infty}_{k=0}\mu_{k}x^{k},
\]
so that 
\[
\Phi(x)=(-x)\frac{\frac{1}{2}}{1+\beta_{1}x+}\underset{j=2}{\stackrel{\infty}{\mathbf{K}}}\frac{-\alpha_{j}x^{2}}{1+\beta_{j}x}=\frac{\alpha_1 x}{1+\beta_{1}x+}\underset{j=2}{\stackrel{\infty}{\mathbf{K}}}\frac{-\alpha_{j}x^{2}}{1+\beta_{j}x}.
\]
If we let 
\[
J_{k}(x):=\frac{-\frac{1}{2}x}{1+\beta_{1}x+}\underset{j=2}{\stackrel{k}{\mathbf{K}}}\frac{\alpha_{j}x^{2}}{1+\beta_{j}x}
\]
be the $k$-th convergent and for any $n$, $(a_{n},b_{n})=(\alpha_{n},\beta_{n})$
is equivalent to for any $k$ large enough (i.e., $k>n+1$,)
\[
J_{k}\left(\frac{1}{n}\right)=MC_{k}(n).
\]

\begin{itemize}
\item First of all, $a_{k}$ and $b_{k}$ are defined to guarantee \eqref{eq:CaoAppr},
namely, as $n\rightarrow\infty$
\[
E_{k}(n):=H_{n}-\log n-\gamma-MC_{k}(n)=O\left(n^{-2k-1}\right);
\]
and meanwhile, the Euler-Maclaurin summation formula also implies
\[
H_{n}-\log n-\gamma-J_{k}\left(\frac{1}{n}\right)=O\left(n^{-2k-1}\right).
\]
This shows 
\[
MC_{k}(n)-J_{k}\left(\frac{1}{n}\right)=O\left(n^{-2k-1}\right).
\]
\item Let $U_{j}(n)$ and $V_{j}(n)$ be the $j$-th denominator tails of
the two continued fractions $MC_{k}(n)$ and $J_{k}(1/n)$, respectively.
Namely, the two recurrences hold
\begin{align}
U_{j}(n) & =n+b_{j}+\frac{a_{j+1}}{U_{j+1}(n)},\label{eq:RecUj}\\
V_{j}(n) & =1+\frac{\beta_{j}}{n}+\frac{\frac{\alpha_{j+1}}{{n^2}}}{V_{j+1}(n)}=1+\frac{\beta_{j}}{n}+\frac{\alpha_{j+1}}{n^{2}V_{j+1}(n)},\label{eq:RecnVj}
\end{align}
where the second recurrence is equivalent to 
\[
nV_{j}(n)=n+\beta_{j}+\frac{\alpha_{j+1}}{nV_{j+1}(n)}.
\]
This further indicates that 
\[
J_{k}\left(\frac{1}{n}\right)=\frac{\frac{a_{1}}{n}}{V_{1}(n)}=\frac{a_{1}}{nV_{1}(n)}=\frac{a_{1}}{n+\beta_{1}+\frac{\alpha_{2}}{nV_{2}(n)}}=\frac{\alpha_{1}}{n+\beta_{1}+}\underset{j=2}{\stackrel{k}{\mathbf{K}}}\frac{\alpha_{j}}{nV_{j}(n)}.
\]
\end{itemize}
The following lemma is one key step, in which we adopted the \emph{big-O
notation}. We write 
\[
f(x)=O(g(x))
\]
as $x\rightarrow a$ or $x\rightarrow\infty$, if, when $x$ is approaching
the limit, $\left|f(x)\right|\leq M\left|g(x)\right|$, for some positive
real number $M$. In the rest of this paper, we will consider (and
sometimes omit) the limit as $n\rightarrow\infty$. 
\begin{lem}
\label{lem:BigOUjANDnVj}For all $1\leq j\leq k$, 
\[
U_{j}(n)=n+b_{j}+O\left(\frac{1}{n}\right)
\]
 and 
\[
nV_{j}(n)=n+\beta_{j}+O\left(\frac{1}{n}\right).
\]
\end{lem}

\begin{proof}
By definition, we first see 
\[
U_{k}(n)=n+b_{k}=n+b_{k}+O\left(\frac{1}{n}\right),
\]
and $V_{k}(n)=1+\beta_{k}/n$ so that
\[
nV_{k}(n)=n+\beta_{k}+O\left(\frac{1}{n}\right).
\]
If $U_{j+1}(n)=n+b_{j+1}+O(1/n)$ for some $1\leq j\leq k-1$, we
first notice 
\[
U_{j+1}(n)=n\left(1+\frac{b_{j+1}}{n}+O\left(\frac{1}{n^{2}}\right)\right).
\]
As $n\rightarrow\infty$, 
\[
\left|\frac{b_{j+1}}{n}+O\left(\frac{1}{n^{2}}\right)\right|<1,
\]
so that we apply the geometric series that 
\[
\frac{1}{1+y}=1-y+O(y^{2})
\]
to see
\[
\frac{1}{U_{j+1}(n)}=\frac{1}{n}\cdot\frac{1}{1+\frac{b_{j+1}}{n}+O\left(\frac{1}{n^{2}}\right)}=\frac{1}{n}\left(1-\frac{b_{j+1}}{n}+O\left(\frac{1}{n^{2}}\right)\right).
\]
Hence, by the recurrence \eqref{eq:RecUj}
\[
U_{j}(n)=n+b_{j}+\frac{a_{j+1}}{U_{j+1}(n)}=n+b_{j}+\frac{a_{j+1}}{n}+O\left(\frac{1}{n^{2}}\right)=n+b_{j}+O\left(\frac{1}{n}\right).
\]
The inductive proof for $nV_{j}(n)$ is similar. 
\end{proof}

\begin{rem}
It is noticeable that we further have 
\begin{equation}
U_{j}(n)=n+b_{j}+\frac{a_{j+1}}{n}+O\left(\frac{1}{n^{2}}\right),\label{eq:UjnBjAj1On2}
\end{equation}
and similarly, 
\begin{equation}
nV_{j}(n)=n+\beta_{j}+\frac{\alpha_{j+1}}{n}+O\left(\frac{1}{n^{2}}\right).\label{eq:nVjnBetajAlaphj1On2}
\end{equation}

Now, we are ready to prove our second main result. Recall that the orthogonal
polynomials $R_{n}(y)$ with respect to $\mu_{k}$, satisfy the three-term
recurrence \eqref{eq:3TermRn}, in which $\alpha_{n}$ and $\beta_{n}$
are defined. 
\end{rem}

\begin{prop}
For any $j\in\mathbb{N}$, $a_{j}=\alpha_{j}$ and $b_{j}=\beta_{j}$.
We also in particular have 
\[
\ensuremath{U_{j}(n)-nV_{j}(n)=O\left(n^{-(2k-2j+1)}\right)}.
\]
\end{prop}

\begin{proof}
The proof by induction begins with $j=1$. Since $a_{1}=\alpha_{1}=1/2$,
from
\[
O\left(n^{-2k-1}\right)=MC_{k}(n)-J_{k}\left(\frac{1}{n}\right)=\frac{a_{1}}{U_{1}(n)}-\frac{\alpha_{1}}{nV_{1}(n)},
\]
we have
\[
\frac{U_{1}(n)-nV_{1}(n)}{U_{1}(n)\cdot nV_{1}(n)}=O\left(n^{-2k-1}\right).
\]
By Lem.~\ref{lem:BigOUjANDnVj}, we see 
\[
U_{1}(n)-nV_{1}(n)=O\left(n^{-2k-1+2}\right)=O\left(n^{-2k+1}\right)=O\left(n^{-(2k-2\cdot1+1)}\right).
\]
Now, by \eqref{eq:UjnBjAj1On2} and \eqref{eq:nVjnBetajAlaphj1On2},
we have 
\[
O\left(n^{-2k+1}\right)=U_{1}(n)-nV_{1}(n)=b_{1}-\beta_{1}+\frac{a_{2}-\alpha_{2}}{n}+O\left(\frac{1}{n^{2}}\right).
\]
If $k$ is large enough, we must have
\[
b_{1}=\beta_{1}\quad\text{and}\quad a_{2}=\alpha_{2}.
\]
Inductively, assume we have $a_{1}=\alpha_{1},\ldots,a_{j}=\alpha_{j}$,
$b_{1}=\beta_{1},\ldots,b_{j-1}=\beta_{j-1}$ for some $j\in\mathbb{N}$
and also 
\[
\ensuremath{U_{t}(n)-nV_{t}(n)=O\left(n^{-(2k-2t+1)}\right)}
\]
for $t=1,2,\ldots,j-1$. By the recurrences \eqref{eq:RecUj} and
\eqref{eq:RecnVj} and the inductive hypothesis,
\[
O\left(n^{-(2k-2(j-1)+1)}\right)=U_{j-1}(n)-nV_{j-1}(n)=a_{j}\left(\frac{1}{U_{j}(n)}-\frac{1}{nV_{j}(n)}\right),
\]
namely, 
\[
\frac{U_{j}(n)-nV_{j}(n)}{U_{j}(n)\cdot nV_{j}(n)}=O\left(n^{-(2k-2(j-1)+1)}\right).
\]
By Lem.~\ref{lem:BigOUjANDnVj}, we see 
\[
U_{j}(n)-nV_{j}(n)=O\left(n^{-(2k-2(j-1)+1)+2}\right)=O\left(n^{-(2k-2j+1)}\right).
\]
Again, by \eqref{eq:UjnBjAj1On2} and \eqref{eq:nVjnBetajAlaphj1On2},
\[
O\left(n^{-(2k-2j+1)}\right)=b_{j}-\beta_{j}+\frac{a_{j+1}-\alpha_{j+1}}{n}+O\left(\frac{1}{n^{2}}\right).
\]
This implies, for $k$ large enough, $b_{j}=\beta_{j}$ and $a_{j+1}=\alpha_{j+1}$. 
\end{proof}

\begin{rem}
This result is compatible with an analytic continuation result. Zagier
\cite[p.~243]{Zagier} defined 
\[
\gamma_{0}(x):=\sum^{\infty}_{n=1}\frac{B_{n}x^{n}}{n}=-\frac{x}{2}+\sum^{\infty}_{k=1}\frac{B_{2k}}{2k}x^{2k}=x\sum^{\infty}_{k=0}\mu_{k}x^{k},
\]
and showed, by analytic continuation, that it satisfies the functional equation
\begin{equation}
\gamma_{0}\left(\frac{x}{1-x}\right)-\gamma_{0}(x)=\log(1-x)+x.\label{eq:FEgamma0}
\end{equation}
Since it is trivial to see 
\[
\frac{x}{1-x}=\frac{1}{n-1}\Leftrightarrow x=\frac{1}{n},
\]
we see that \eqref{eq:FEgamma0} indicates, for $n\geq2$, 
\begin{equation}
\gamma_{0}\left(\frac{1}{n-1}\right)-\gamma_{0}\left(\frac{1}{n}\right)=\log\left(1-\frac{1}{n}\right)+\frac{1}{n}=\log(n-1)-\log n+\frac{1}{n}.\label{eq:RecurrGammaZero}
\end{equation}
Summation by telescoping yields, for $n\geq2$, 
\begin{equation}
\gamma_{0}(1)=\gamma_{0}\left(\frac{1}{n}\right)+H_{n}-\log n -1.\label{eq:Gammo01}
\end{equation}
Here, 
\begin{align*}
\gamma_{0}\left(\frac{1}{n}\right) & =x\sum^{\infty}_{k=0}\mu_{k}x^{k}\bigg|_{x=\frac{1}{n}}=-\frac{\alpha_{1}\frac{1}{n}}{1+\beta_{1}\frac{1}{n}+}\underset{j=2}{\stackrel{\infty}{\mathbf{K}}}\frac{\alpha_{j}\frac{1}{n^{2}}}{1+\frac{\beta_{j}}{n}}\\
 & =-\frac{\alpha_{1}}{n+\beta_{1}+\frac{\alpha_{2}}{n+\beta_{1}+\frac{\alpha_{3}}{n+\beta_{2}+\ddots}}},
\end{align*}
and letting $n\rightarrow\infty$ in \eqref{eq:Gammo01}
and noting $\gamma_{0}(0)=0$, we have $\gamma_{0}(1)=\gamma - 1$. Consequently, 
\[
H_n - \log n -\gamma =-\gamma_0\left(\frac{1}{n}\right) = \Phi\left(\frac{1}{n}\right).
\]
\end{rem}
This is compatible with the extra factor $-1$ in $\alpha_1=-\mu_0$. 
\section*{Acknowledgment}

We would like to thank the Duke Kunshan University Undergraduate Summer
Research Scholars Program of Summer 2026, especially for approving
the project proposal by Lin Jiu and providing financial support
for Yihang Yin. During this time, the major results of this
work were obtained. 

Yihang Yin also acknowledges the use of OpenAI Codex using the GPT-5 model during
the exploratory and proof-development stages of this work. All AI-generated
suggestions were substantially revised, corrected, and independently verified
by the authors, who assume full responsibility for the mathematical content. The use of AI
is fully in accordance with editorial standards on the responsibility, ethnicity, and transparency.

\end{document}